\documentclass[12pt]{amsart}
\usepackage{amsmath, amssymb, amsthm}
\usepackage{mathtools}
\usepackage{array}
\usepackage{tikz-cd}
\usepackage{tabularx}
\usepackage{mathrsfs}
\usepackage{enumitem}
\usepackage[backend=biber,
            style=alphabetic,
            citestyle=alphabetic]{biblatex}

\usepackage[T1]{fontenc}
\usepackage{hyperref}
\usepackage{color}
\usepackage[margin=1in]{geometry}

\usepackage{xspace}
\usepackage{silence}
\title{Chow Rings of Vector Space Matroids}
\author{Thomas Hameister}
\address{University of Wisconsin--Madison}
\email{thameister@wisc.edu}
\author{Sujit Rao}
\address{Cornell University}
\email{sr869@cornell.edu}
\author{Connor Simpson}
\address{Cornell University}
\email{cgs93@cornell.edu}
\bibliography{../chow_rings.bib}

\theoremstyle{plain}
\newtheorem{thm}{Theorem}[section]
\newtheorem{lem}[thm]{Lemma}
\newtheorem{prop}[thm]{Proposition}
\newtheorem{cor}[thm]{Corollary}

\newtheorem{conj}[thm]{Conjecture}

\theoremstyle{remark}
\newtheorem{rmk}[thm]{Remark}
\newtheorem{ex}[thm]{Example}

\theoremstyle{definition}
\newtheorem{defn}[thm]{Definition}

\DeclareMathOperator{\rank}{rank}

\DeclareMathOperator{\inv}{inv}
\DeclareMathOperator{\des}{des}
\DeclareMathOperator{\exc}{exc}
\DeclareMathOperator{\maj}{maj}
\DeclareMathOperator{\CD}{CD}
\DeclareMathOperator{\fix}{fix}

\newcommand{\FF}{\ensuremath{\mathbb{F}}}
\newcommand{\NN}{\ensuremath{\mathbb{N}}}
\newcommand{\QQ}{\ensuremath{\mathbb{Q}}}
\newcommand{\ZZ}{\ensuremath{\mathbb{Z}}}

\renewcommand{\emptyset}{\varnothing}
\newcommand{\setof}[2]{\left\{ #1 \,:\, #2 \right\}}
\newcommand{\idealof}[2]{\left( #1 \,\colon\, #2 \right)}
\newcommand{\iso}{\cong} 
\newcommand{\rr}{\ensuremath{\mathbf{r}}}
\renewcommand{\subset}{\subseteq}
\newcommand{\qnom}{\genfrac{[}{]}{0pt}{}}
\newcommand{\eulernom}{\genfrac{<}{>}{0pt}{}}

\newcommand{\LL}{L}
\newcommand{\hilb}{H}
\newcommand{\hilbfun}{h}
\newcommand{\hpoly}{h}
\newcommand{\word}[1]{\emph{#1}}
\newcommand{\Sym}{\mathfrak{S}}

\DeclareMathOperator{\cl}{cl}
\DeclareMathOperator{\sech}{sech}

\newcommand{\atomic}{atomic\xspace}
\begin{document}
\newcommand{\cmt}[1]{\qquad}
\newcommand{\com}[1]{\qquad}
\newcommand\phantomarrow[2]{%
  \setbox0=\hbox{$\displaystyle #1\to$}%
  \hbox to \wd0{%
    $#2\mapstochar
     \cleaders\hbox{$\mkern-1mu\relbar\mkern-3mu$}\hfill
     \mkern-7mu\rightarrow$}%
  \,}
  
\begin{abstract}
	The \word{Chow ring of a matroid} (or more generally, \atomic latice) is an invariant whose importance was demonstrated by Adiprasito, Huh and Katz, who used it to resolve the long-standing Heron-Rota-Welsh conjecture.
	Here, we make a detailed study of the Chow rings of uniform matroids and of matroids of finite vector spaces. In particular, we express the Hilbert series of such matroids in terms of permutation statistics; in the full rank case, our formula yields the $\maj$-$\exc$ $q$-Eulerian polynomials of Shareshian and Wachs.
We also provide a formula for the Charney-Davis quantities of such matroids, which can be expressed in terms of either determinants or $q$-secant numbers.
\end{abstract}

\maketitle

\section{Introduction}
Since Stanley's 1975 proof of the upper bound conjecture for simplicial spheres via the Stanley-Reisner ring, the study of graded rings associated to combinatorial objects has yielded many deep insights into combinatorics (and vice versa).
The \emph{Chow ring} of an \atomic lattice, defined by Feichtner and Yuzvinsky in \cite{fy} is the latest instance of the pattern. 

The power of Feichtner and Yuzvinsky's construction was demonstrated by Adiprasito, Huh, and Katz, who applied a slight variation of it to the lattice of flats of a matroid in order to resolve the long-standing Heron-Rota-Welsh conjecture.
Along the way, they also show that Chow rings arising from geometric lattices satisfy Poincar\'e duality and versions of the hard Lefschetz theorem and the Hodge-Riemann relations.
Here, we explore some of Chow rings' combinatorial structure.

\subsubsection*{Organization}
In the remainder of this section, we summarize some of our main results; 
Section \ref{sec:background} contains the definitions of matroids and Chow rings. In Section \ref{sec:hilbert}, we derive an explicit form (in terms of permutation statistics) for the Hilbert series of the Chow ring of the matroid associated to a finite vector space. 
The Charney-Davis quantities of such matroids are computed in Section \ref{sec:charneydavis}.
In Section \ref{sec:uniform} we state the specializations of our results to the case of uniform matroids.
Finally, in Section \ref{sec:conjectures} we present conjectures and ideas for further work.

\subsection{Summary of main results}
Let $\FF_q$ be the finite field of order $q$. Associated to the finite vector space $\FF_q^n$ is the matroid $M_r(\FF_q^n)$ whose independent sets are linearly independent subsets of $\FF_q^n$ of size at most $r$.
The lattice of flats of $M_r(\FF_q^n)$ is given by the collection of subspaces of $\FF_q^n$ of dimension at most $r$ ordered by inclusion together with the maximal subspace $\FF_q^n$.

In addition, let $U_{n,r}$ denote the uniform matroid of rank $r$ on ground set $[n] := \{1,2,\ldots,n\}$.  The lattice of flats of $U_{n,r}$ consists of all subsets of $[n]$ of size at most $r$, together with $[n]$, all ordered by inclusion.
Finally, for any matroid $M$, let $A(M)$ be the Chow ring of $M$, and let $\hilb(A(M_r(\FF_q^n)),t)$ be the Hilbert series of $A(M_r(\FF_q^n))$ (defined in Section \ref{sec:backgroundAlgebra}).
\begin{thm} \label{cor:linearhilbert} 
For $r = 1,\dots, n$ the Hilbert series of $A\big( M_r(\FF_q^n) \big)$ is given by
\begin{equation}
\hilb\big( A(M_r(\FF_q^n)),t \big)
= \sum_{\sigma\in \Sym_n}q^{\maj(\sigma) - \exc(\sigma)}t^{\exc(\sigma)}
- \sum_{j=r}^{n-1} \sum_{\sigma\in F_{n,n-j}}q^{\maj(\sigma) - \exc(\sigma)}t^{r-\exc(\sigma)}
\label{qRankHilb}
\end{equation}
 where $F_{n,n-j}$ is the set of permutations in $\Sym_n$ with at least $n-j$ fixed points.
\end{thm}
In particular, when $r = n$, the Hilbert series of $A\big( M_n(\FF_q^n) \big)$ is
\[
	\hilb\Big( A\big( M_n(\FF_q^n) \big), t\Big) = \sum_{\sigma\in \mathcal \Sym_n}q^{\maj(\sigma) - \exc(\sigma)}t^{\exc(\sigma)} = A_n(q,t),
\]
the $n$th $\maj$-$\exc$ $q$-Eulerian polynomial considered by Shareshian and Wachs in \cite{sw}.

We also study the Charney-Davis quantity of $A(M_r(\FF_q^n))$, defined as $(-1)^{\frac{r-1}{2}}\hilb(A(M_r(\FF_q^n)), -1)$ for odd $r$ (see Section \ref{sec:backgroundAlgebra}). When $r$ is even, the Charney-Davis quantity vanishes (see Remark \ref{rmk:evencd}). When $r$ is odd, the Charney-Davis quantity has an interpretation in terms of the signature of a quadratic form on the Chow ring (see Remark \ref{rmk:signature}), and in this case, we derive two formulas for the for the Charney-Davis quantity, one in terms of determinants and one in terms of the $q$-secant numbers.\begin{thm}
    \label{thm:linearcd}
\begin{enumerate}[label=(\alph*)]
\item For odd $r$, the Charney-Davis quantity of $A\big( M_r(\FF_q^n) \big)$ is
\[  (-1)^{\frac{r-1}{2}}\sum_{k=0}^{\frac{r-1}{2}}\qnom{n}{2k}_qE_{2k,q}\]
where $E_{2k,q}$ is the $q$-analogue of the $k$-th secant number (see Definition \ref{defn:ts}).
\item More explicitly, for odd $r$ the Charney Davis quantity in part (a) is equal to
\[
(-1)^{\frac{r-1}{2}}\left(1+[n]_q!\sum_{a = 1}^{\frac{r-1}{2}} \frac{(-1)^a}{[n-2a]_q!}\Delta_{a,q}\right)
\]
for $\Delta_{a,q}$ the determinant
\[
\Delta_{a,q} = \det\left( \begin{array}{ccccc} \frac{1}{[2]_q!} & 1 & 0 & \cdots & 0\\ \frac{1}{[4]_q!} & \frac{1}{[2]_q!} & 1 & \cdots & 0\\
\vdots & \vdots & \vdots & \ddots & \vdots \\
\frac{1}{[2a-2]_q!} & \frac{1}{[2a-4]_q!} & \frac{1}{[2a-6]_q!} & \cdots& 1\\ \frac{1}{[2a]_q!} & \frac{1}{[2a-2]_q!} & \frac{1}{[2a-4]_q!} & \cdots& \frac{1}{[2]_q!} \end{array} \right).
\]
\end{enumerate}
\end{thm}

All of these invariants are $q$-analogs of the corresponding invariants of the Chow ring of the uniform matroid.

\section{Definitions and Background}
\label{sec:background}
In this section, we first define the Charney-Davis quantity. We then define Chow rings and state some salient results on them.
Finally, we give a brief review of some permutation statistics, which we use to establish notation and introduce some of the $q$-analogs that will later appear. For an introduction and reference about matroid theory, we refer the reader to \cite{oxley}.

\subsection{Hilbert Series and the Charney-Davis Quantity}
\label{sec:backgroundAlgebra}
Let $R$ be an $\mathbb N$-graded $\ZZ$-algebra with the property that for all $d \in \mathbb N$, the degree-$d$ homogeneous component $R_d$ of $R$ is a torsion-free $\ZZ$-module.
We can then define the \word{Hilbert function} of $R$ by $h(R, d) \coloneqq \dim_\ZZ R_d$ and the \word{Hilbert series} of $R$ by $\hilb(R,t) \coloneqq \sum_{d \in \mathbb N} h(R,d) t^d$.

The Hilbert series of some rings, including those that we will study, are \word{symmetrical}, meaning that there exists an $r \geq 0$ such that $\hilbfun(R, d) = 0$ for $d > r$, $\hilbfun(R, r) \neq 0$, and $\hilbfun(R, d) = \hilbfun(R, r-d)$ for all $0 \leq d \leq r$.

When the Hilbert series of $R$ is a polynomial of degree $r$, we call the number  
\[ \CD(R) := \begin{cases} 
(-1)^{r/2} \hilb(R,-1), & r \textrm{ even} \\
\hilb(R,-1), & r \textrm{ odd} \end{cases}
\]
the \word{Charney-Davis quantity} of $R$. In particular, if $R$ has symmetric Hilbert series of odd degree, then $\CD(R) = 0$. 
The Charney-Davis quantity was introduced in \cite{cdquant} and is related to a conjecture of Charney and Davis for posets associated to flag simplicial complexes. See \cite{athansurvey} for a more recent framework towards approaching questions stemming from Charney and Davis' original conjecture.
For an alternative interpretation of the Charney-Davis quantity in the context of the Chow ring of a matroid, see Remark \ref{rmk:signature}.

\subsection{Chow Rings of Matroids}
\label{sec:chow}
Let $M$ be a finite matroid on ground set $E$; that is, a pair $(E, \mathcal{I})$ where $\emptyset \subsetneq \mathcal{I} \subset 2^E$ is the collection of \word{independent sets of $M$} and satisfies
\begin{enumerate}
	\item $A \in \mathcal{I} \implies 2^A \subset \mathcal{I}$, and 
\item if $A,B \in \mathcal{I}$ with $\#A  > \#B$ then there exists $x \in A \setminus B$ such that $B \cup \{x\} \in \mathcal{I}$.
\end{enumerate}
The \word{rank} of $S \subset E$ is the size of any maximal independent subset of $S$, and the \word{closure} of $S$ is $\cl(S) \coloneqq \setof{x \in E}{\rank(S \cup \{x\}) = \rank(S)}$. We will call $S$ a \word{flat} if $\cl(S) = S$.
The flats of $M$, ordered by inclusion, form a geometric lattice $L = L(M)$ called the \word{lattice of flats} of $M$.
We will write $\bot$ for the minimal flat of $M$, and $\top$ for the maximal flat of $M$.
\begin{defn}
   The \word{Chow ring} of $M$ on ground set $E$ with lattice of flats $L$ is 
   \[ A(L) \coloneqq A(M) \coloneqq \ZZ[x_F\,\colon F \in L(M) \setminus \{\bot\} ] / (I_1 + I_2) \]
   where $I_1$ and $I_2$ are the ideals with generators
   \begin{align*}
       I_1 &= \idealof{ x_F x_G}{\textrm{$F$ and $G$ are incomparable}} \\
       I_2 &= \idealof{ \sum_{i \in F \in L(M)} x_F}{i \in E}
   \end{align*}
\end{defn}

Each homogeneous component of a Chow ring is a torsion-free $\ZZ$-module (see Cor. 1 in \cite{fy}), so we may speak of its Hilbert function and Hilbert series as a $\ZZ$-algebra, as defined in Section \ref{sec:backgroundAlgebra}.
We now state some results on Chow rings of matroids that we will make use of later in the paper.
\subsubsection{Gr\"obner Basis and Hilbert Series}
Feichtner and Yuzvinsky found a Gr\"obner basis for this ring and proved the following theorem about its Hilbert series in \cite{fy}.
\begin{thm}[\cite{fy} Corollary 2] 
    \label{thm:fyhilbert}
    The Hilbert series of $A(L)$ is
	\[\hilb(A(L), t) = 1 + \sum _{\bot = F_{0} < F_{1} < \dots < F_{m}} \prod _{i=1} ^{m} \frac{t(1 - t^{\rank F_{i} - \rank F_{i-1} - 1})}{1 - t}.\]
	where the sum is taken over all chains of flats $\bot = F_0 < F_1 < \cdots < F_m$ in $L$. In particular, the Hilbert function is given combinatorially as follows.
	\[ \dim A(L)_k = \#\setof{ x_{F_1}^{\alpha_1}\cdots x_{F_\ell}^{\alpha_\ell} }{ { 1\leq \alpha_i\leq {\rm rk}(F_i) - {\rm rk}(F_{i+1}) - 1 , \,\,\, \sum
	\alpha_i = k} }  \]
	where the set on the right ranges over all flats $F_1>\cdots >F_\ell$ in $L(M)$.
\end{thm}
\subsubsection{Poincar\'e duality}
Adiprasito, Huh, and Katz show Chow rings of matroids satisfy a form of Poincar\'e duality.

\begin{thm}[Poincar\'e duality; c.f. \cite{ahk} Theorem 6.19] 
Let $M$ be a matroid of rank $r$.  For $q \leq r-1$, the multiplication map 
\[ A^q(M)\times A^{r-1-q}(M)\to A^{r-1}(M) \]
defines an isomorphism
\[  A^{r-1-q}(M)\iso {\rm Hom}_\ZZ(A^q(M),A^{r-1}(M)) \]
\label{Poincare1}
\end{thm}
\begin{rmk}
  It is an immediate consequence of Corollary 6.11 of \cite{ahk} that $A^{r-1}(M) \iso \ZZ$. Hence, Theorem \ref{Poincare1} implies that $\dim_\ZZ A^{r-1-q}(M) = \dim_\ZZ A^q(M)$.
	This shows that $A(M)$ has a symmetrical Hilbert series.
If we speak of the Hilbert series or Charney-Davis quantity of a matroid $M$, then we are referring to  that of its Chow ring $A(M)$.
\end{rmk}
\begin{rmk}
  \label{rmk:signature}
  Since $A^{r-1}(M) \iso \ZZ$, when $r$ is odd, the squaring map $Q: A^{(r-1)/2}(M) \times A^{(r-1)/2}(M) \to A^{r-1}(M)$ with $Q(x) = x^2$ defines a quadratic form on $A^{(r-1)/2}(M)$. By Theorem 1.1 of \cite{leungreiner}, the fact that the Hodge-Riemann relations hold for $A(M)$ implies that the signature of this quadratic form is equal to the Charney-Davis quantity of $A(M)$.
\end{rmk}
\subsection{Permutation Statistics and Polynomials}
In this section, we will establish notation for permutation statistics. We will also discuss Eulerian polynomials, which will appear when we examine the Hilbert series of Chow rings, and the tangent-secant numbers, which will appear when we examine the Charney-Davis quantities.

Let $\Sym_n$ denote the symmetric group on $n$ letters.
\begin{defn}
Let $\sigma \in \Sym_n$ be a permutation. Then, define the statistics
\begin{align*}
\inv(\sigma) &= \#\setof{(i,j)}{\sigma(i)>\sigma(j)} \\
\des(\sigma) &= \#\setof{i\in [n-1]}{\sigma(i+1)<\sigma(i)} \\
\exc(\sigma) &= \#\setof{i\in [n]}{\sigma(i)>i} \\
\maj(\sigma) &= \sum_{i,\; \sigma(i)<\sigma(i+1)} i \\
\end{align*}
\end{defn}
\subsubsection{Eulerian polynomials}
The Eulerian polynomials and their $q$-analogs appear in the Hilbert series of the matroids that we study. To motivate the $q$-analogs, we first review the classical Eulerian polynomials.
\begin{defn}
The Eulerian polynomial $A_n(t)$ is the polynomial
\[
A_n(t) = \sum_{\omega\in \Sym_n} t^{\exc(\omega)}
\]
\end{defn}
These polynomials have many interesting applications; see \cite{pk} for further exposition.  The polynomials $A_n(t)$ satisfy the following identities
\begin{prop}[\cite{pk} Theorem 1.4] 
	\[ \displaystyle A_n(t) = \sum_{k = 0}^{n-1}\binom{n}{k}A_k(t)(t+1)^k \]
\end{prop}
\begin{prop}[\cite{pk} Theorem 1.6] 
	\label{prop:eulerexp}
	The exponential generating function of the polynomials $A_n(t)$ is
\[
\sum_{n\geq 0}A_n(t)\frac{x^n}{n!} = \frac{t-1}{t-e^{z(t-1)}}.
\]
\end{prop}
\noindent The coefficient of $t^k$ in $A_n(t)$ is the \word{$n$-th Eulerian number} and is written
\[ A(n,k) \coloneqq \eulernom{n}{k} \coloneqq \#\setof{\sigma\in \Sym_n}{\exc(\sigma) = k}.\]

\noindent Now, we discuss the $\maj$-$\exc$ $q$-Eulerian polynomials of Shareshian and Wachs.
\begin{defn}
	The $n$th $\maj$-$\exc$ \emph{$q$-Eulerian polynomial} (or merely $q$-Eulerian polynomial) $A_n(q,t)$ is the polynomial
\[
A_n(q,t) \coloneqq A_n^{\maj,\exc}(q,tq^{-1}) = \sum_{\sigma\in \Sym_n}q^{\maj(\sigma) - \exc(\sigma)}t^{\exc(\sigma)}
\]
As above, define the $q$-Eulerian number $\eulernom{n}{j}_q$ to be the coefficient of $t^j$
\[
\eulernom{n}{j}_q \coloneqq \sum_{\substack{\sigma\in \Sym_n \\ \exc(\sigma) = j}} q^{\maj(\sigma) - \exc(\sigma)} =  \sum_{\substack{\sigma\in \Sym_n \\ \exc(\sigma) = j}} q^{\maj(\sigma) - j}
\]
\end{defn}
\noindent The following theorem gives a $q$-analog of Proposition \ref{prop:eulerexp}.
\begin{thm}[\cite{sw}, Thm 1.1]
	\label{thm:swexp}
The $q$-Eulerian polynomials $A_n(q,t)$ are the unique polynomials with $q$-exponential generating function
\[
\sum_{n\geq 0} A_n(q,t)\frac{x^n}{[n]_q!} =\frac{(t-1)e_q(x)}{te_q(x) - e_q(tx)}
\]
where $e_q(x) \coloneqq \sum_{n \geq 0} \frac{x^n}{[n]_q!}$ is the $q$-exponential function.
\end{thm}
\subsubsection{Tangent-Secant numbers}
The tangent-secant numbers and a $q$-analog of them will appear in our investigation of Charney-Davis quantities.
\begin{defn}
	The \emph{$n$-th tangent-secant number} $E_{n}$ is the coefficient of $\frac{x^n}{n!}$ in the exponential generating function
\[
\tanh(x)+\sech(x) = \sum_{n\geq 0}E_n\frac{x^n}{n!}
\]
\end{defn}
\begin{rmk}
	In the literature, the numbers $E_{2n}$ are often referred to as the Euler numbers. To avoid confusion with the Eulerian numbers, we will refrain from using this language. Instead, we call the numbers $E_{2n}$ the \word{secant numbers} and the numbers $E_{2n+1}$ the \word{tangent numbers}.
The nomenclature that we use is justified by the observation that, since $\tanh(x)$ is odd and $\sech(x)$ even,
\[
\tanh(x) = \sum_{n\geq 0} E_{2n+1}\frac{x^{2n+1}}{(2n+1)!}\text{\quad and\quad}\sech(x) = \sum_{n\geq 0}E_{2n}\frac{x^{2n}}{(2n)!}.
\]
Hence, 
\[
\tan(x) = \sum_{n\geq 0}(-1)^nE_{2n+1}\frac{x^{2n+1}}{(2n+1)!}\text{\quad and\quad}\sec(x) = \sum_{n\geq 0}(-1)^nE_{2n}\frac{x^{2n}}{(2n)!}.
\]
\end{rmk}
In Section \ref{sec:charneydavis}, we will also prove $q$-analogues of the following.
\begin{prop}[\cite{stanley:altPerms}, equation 1.8]
	\label{prop:sectandet}
For all $n$, we have $E_{2n} = (-1)^n(2n)!\Delta_n$ for the following determinant
\[
\Delta_n = \det\left( \begin{array}{ccccc} \frac{1}{2!} & 1 & 0 & \cdots & 0\\ \frac{1}{4!} & \frac{1}{2!} & 1 & \cdots & 0\\
\vdots & \vdots & \vdots & \ddots & \vdots \\
\frac{1}{(2n-2)!} & \frac{1}{(2n-4)!} & \frac{1}{(2n-6)!} & \cdots& 1\\
\frac{1}{(2n)!} & \frac{1}{(2n-2)!} & \frac{1}{(2n-4)!} & \cdots& \frac{1}{2!} \end{array} \right)
\]
\end{prop}
\begin{prop}[cf. \cite{eulernos}]
For all $n$, $\displaystyle E_{2n} = -\sum_{k=0}^{n-1}\binom{2n}{2k}E_{2k}$.
\end{prop}
To define the \word{$q$-tangent-secant numbers}, let
\begin{align*}
	\sinh_q(t) &\coloneqq \sum_{n\geq 0} \frac{t^{2n+1}}{(q;q)_{2n+1}} & \cosh_q(t) &\coloneqq \sum_{n\geq 0} \frac{t^{2n}}{(q;q)_{2n}} \\
	\sech_q(t) &\coloneqq \frac{1}{\cosh_q(t)} & \tanh_q(t) &\coloneqq \frac{\sinh_q(t)}{\cosh_q(t)}
\end{align*}
where $(t;q)_n = (1-t)(1-tq)\cdots (1-tq^{n-1})$ is the \emph{Pochhamer symbol}. 
\begin{defn}
\label{defn:ts}
The \word{$n$-th $q$-tangent-secant number}, $E_{n,q}$, is the coefficient of $t^n$ in the generating function 
\[\sech_q(t) + \tanh_q(t) = \sum_{n\geq 0} E_{n,q}\frac{t^n}{(q;q)_n}.\]
\end{defn}
Up to signs, the tangent-secant numbers in Definition \ref{defn:ts} agree with those studied in the work of Foata and Han and of Josuat-Verg{\`e}s in \cite{foata} and \cite{jv}, respectively.
\begin{rmk}
In the case $q=1$, $E_{n,q} = E_{n}$ is the classical $n$th tangent/secant number.
\end{rmk}

\section{Hilbert series of vector space matroids}
\label{sec:hilbert}
The main results of this section will be Theorem \ref{cor:linearhilbert}, the expression of the Hilbert series in terms of $q$-Eulerian polynomials, and the resulting specialization to the uniform matroid.
\subsection{Method for calculating Hilbert series of Chow rings}
\label{sec:methods}
\let\olddim\dim
\renewcommand{\dim}{\ensuremath{\operatorname{\olddim_{\ZZ}}}}
{
We begin by deriving a useful recurrence for the Hilbert series of the Chow ring of a matroid.
The technique we present below makes use of Theorem \ref{thm:fyhilbert} covered above to give a formula for the Hilbert series of any geometric lattice $L$ of rank $r+1$ with the property 
\begin{equation}
\label{eq:UpperBdProp}
[Z, \top] \iso [Z', \top]\text{ for all }Z, Z' \in L\text{ with }\rank(Z) = \rank(Z').
\tag{$*$}
\end{equation} 
In the following, we assume that $L$ is such a lattice. 
\begin{prop}
If $L$ is a geometric lattice such that property \eqref{eq:UpperBdProp} holds and $(Z_1, \ldots, Z_r)$ is a sequence of elements of $L$ with $\rank(Z_i) = i$ for all $i$, then
\[ \hilb(A(L),t) = [r+1]_t + t \sum_{i = 2}^r |L_i|\,  [i-1]_t \, \hilb( A([Z_i, \top]), t).  \label{eq:recurrence} \]
\end{prop}
\begin{proof}
From Theorem \ref{thm:fyhilbert}, we have
\[ \dim A^q(L) = \#\setof{ x_{F_1}^{\alpha_1}\cdots x_{F_\ell}^{\alpha_\ell} }{ { 1\leq \alpha_i\leq {\rm rk}(F_i) - {\rm rk}(F_{i+1}) - 1 , \,\,\, \sum
	\alpha_i = q} }  \]
where $F_1 > F_2 > \cdots > F_\ell$ ranges over all chains of elements of $L$.
For each $2\leq j\leq r$, define 
\[  N_{q,j}\coloneqq \#\setof{ x_{F_1}^{\alpha_1}\cdots x_{F_\ell}^{\alpha_\ell} }{ { 1\leq \alpha_i\leq {\rm rank}(F_i) - {\rm rank}(F_{i+1}) - 1 , \,\,\, \sum
	\alpha_i = q,\,\,\, {\rm rank}(F_1) = j} } \]
Then $\dim A^q(L) = \sum_{j=2}^{r+1} N_{q,j}$. Now for each $2\leq j\leq r$, property \eqref{eq:UpperBdProp} implies
\begin{align*}
N_{q,j} &= \# L_j\cdot\#\setof{ x_{Z_j}^{\alpha_1}x_{F_2}^{\alpha_2}\cdots x_{F_\ell}^{\alpha_\ell} }{\substack{Z_j = F_1>F_2>\cdots>F_\ell, \\ { 1\leq \alpha_i\leq {\rm rk}(F_i) - {\rm rk}(F_{i+1}) - 1 , \,\,\, \sum
	\alpha_i = q} }}\\
	& =  \#L_j \cdot \sum_{p=1}^{j-1}\#\setof{ x_{Z_j}^{p}x_{F_2}^{\alpha_2}\cdots x_{F_\ell}^{\alpha_\ell} }{\substack{Z_j = F_1>F_2>\cdots>F_\ell \\ { 1\leq \alpha_i\leq {\rm rk}(F_i) - {\rm rk}(F_{i+1}) - 1 , \,\,\, \sum_{i=2}^\ell 
	\alpha_i = q-p}} }\\
	& = \#L_j \cdot \sum_{p=1}^{j-1}\dim A^{q-p}([Z_j,\top])
\end{align*}
While $N_{q,r+1} = \#\{x_\top^q\} = 1$. Hence, we have
\[\label{ahkrec}
    \dim A^q(L)  =  1 + \sum_{i = 2}^r |L_i| \sum_{p=1}^{i-1}  \dim A^{q-p}([Z_i, \top]) .\]
    This recurrence for the dimension of a homogeneous component can be lifted to a recurrence for the Hilbert series of $A(L)$ in the following manner. For a fixed $0 \leq k \leq r-1$, let $(Z_1, \ldots, Z_r)$ be a sequence of elements of $L$ with $\rank(Z_i) = i$ for all $i$. Then
\begin{align*}
    \hilb(L,t) &= \sum_{q = 0}^r \dim A^q(L)\; t^q   \\
    	   &= \sum_{q = 0}^r \left( 1 + \sum_{i = 2}^r \#L_i \cdot \sum_{p=1}^{i-1}   \dim A^{q-p}([Z_i, \top]) \right) t^q \\
	   &= [r+1]_t + \sum_{i = 2}^r  \#L_i \cdot \sum_{p=1}^{i-1} \sum_{q = 0}^r   \dim A^{q-p}([Z_i, \top]) \; t^q
\end{align*}
Since $\dim A^{q-p}([Z_i,\top]) = 0$ when $q-p < 0$ by convention, the innermost sum above really only runs from $q = p$ to $q = r$. Making this change and setting $k = q-p$, we can rewrite the above as
\[ [r+1]_t + \sum_{i = 2}^r \#L_i \cdot \sum_{p = 1}^{i-1} t^p \sum_{k = 0}^{r-p}  \dim A^k([Z_i, \top]) \; t^k. \]
Now, observe that $\rank([Z_i, \top]) = r + 1 - i$ and that $p \leq i-1$, so $r-p \geq r - i + 1$.
Hence, $\sum_{k = 0}^{r-p} \dim A^k([Z_i, \top]) t^k = \hilb([Z_i, \top], t)$ for every $p$ and $i$, so we obtain the proposition.
\end{proof}
}

We will now state the recurrence for the Hilbert series that one gets by applying Proposition \eqref{eq:recurrence} to matroids of special interest.

\subsubsection*{Uniform matroids}
Each upper interval of $\LL(U_{n, r+1})$ is the lattice of flats of a uniform matroid on a smaller ground set and of lower rank. Hence
\[ \hilb( A(U_{n, r+1}), t) = [r+1]_t + t \sum_{i = 2}^r \binom{n}{i} [i-1]_t\, \hilb( A(U_{n-i, r+1-i}),t). \]
In particular, if we define $A(U_{0,0}) = \ZZ$, then for the case $r = n-1$ we have
\begin{align*}
  \hilb( A(U_{n, n}), t) 
  &= [n]_t + t \sum_{i = 2}^{n-1} \binom{n}{i} [i-1]_t\, \hilb( A(U_{n-i, n-i}),t)  \\
  &= 1 + t \sum_{i = 1}^{n} \binom{n}{i} [i-1]_t\, \hilb( A(U_{n-i, n-i}),t). 
\end{align*}
\subsubsection*{Subspaces of vector spaces over finite fields}
The formula for vector spaces over finite fields is a $q$-analog of the one for the uniform matroid.
\[ \hilb\left( A\big(M_{r+1}(\FF^n_q)\big), t\right) = [r+1]_t + t \sum_{i = 2}^r  [i-1]_t \, \qnom{n}{i}_q \hilb\left( A\big( M_{r+1-i}( \FF_q^{n-i}) \big), t\right) \]
In particular, if we write $M(\FF_q^n) = M_n(\FF_q^n)$ and set $A(M(\FF_q^0)) = \ZZ$, then similar to the uniform case, for $r = n-1$,
\begin{equation} \hilb\left( A\big(M(\FF^n_q)\big), t\right) = 1 + t \sum_{i = 1}^n  [i-1]_t \, \qnom{n}{i}_q \hilb\left( A\big( M( \FF_q^{n-i}) \big), t\right) \label{FqRec} \end{equation}
\subsection{Full-rank vector space matroid}
\label{sec:subspace}
Write $M(\FF_q^n) = M_{n}(\FF_q^n)$.
	The main result of this section is a proof  that the Hilbert series of $A\big(M(\FF_q^n)\big)$ is the $\maj$-$\exc$ $q$-Eulerian polynomial of \cite{sw}.
	We also find a new recurrence for the $q$-Eulerian polynomials.

	To characterize the Hilbert series of $A\big( M(\FF_q^n) \big)$, we compute its $q$-exponential generating function.
\begin{lem}
Define $h_0 \coloneqq 1$. The $q$-exponential generating function of $h_n(t) \coloneqq \hilb\Big(A\big(M(\FF_q^n)\big),t\Big)$ is given by
\[ F(t,x) \coloneqq \sum_{n\geq 0}h_n(t)\frac{x^n}{[n]_q!} = \frac{(t-1)e_q(t)}{te_q(t) - e_q(tx)} \]
where $e_q$ denotes the $q$-exponential function $e_q(x) \coloneqq \sum_{n\geq 0}\frac{x^n}{[n]_q!}$.
\label{qegf}
\end{lem}
\begin{proof}
By equation \eqref{FqRec}, we have the relation
\[ h_n = 1+t\sum_{i=1}^{n}[i-1]_t\qnom{n}{i}_q h_{n-i} \]
Then, the generating function $F(t,x)$ satisfies
\begin{align*}
F(t,x) &= 1+\sum_{n\geq 1}\frac{x^n}{[n]_q!}+t\sum_{n\geq 1}\sum_{i = 1}^n\left( [i-1]_t\qnom{n}{i}_qh_{n-i} \right)\frac{x^n}{[n]_q!}\\
& = e_q(x) + t\sum_{n\geq 1}\sum_{i = 1}^n\left( [i-1]_t\frac{x^i}{[i]_q!} \right)\left( h_{n-i}\frac{x^{n-i}}{[n-i]_q!} \right)\\
& = e_q(x)+tF(t,x)G(t,x)
\end{align*}
for $G(t,x) = \sum_{i\geq 1}[i-1]_t\frac{x^i}{[i]_q!}$. We can rewrite $G(t,x)$ as
\begin{align*}
	G(t,x) &= \frac{1}{t - 1}\sum_{i\geq 1}(t^{i-1} - 1)\frac{x^i}{[i]_q!}
	= \frac{1}{t - 1}\left(\frac{e_q(tx)-1}{t} - e_q(x) + 1\right) \\
	&= \frac{1}{t^2 - t}\Big(e_q(tx) - te_q(x) + t - 1\Big)
\end{align*}
Substituting into the equation above and solving for $F$, we get
\[
F(t,x) = \frac{e_q(x)}{1 - \frac{1}{t-1}\Big(  e_q(tx) - te_q(x)\Big)} = \frac{(t-1)e_q(x)}{te_q(x) - e_q(tx)}\qedhere
\]
\end{proof}
\begin{cor}
	\label{cor:fullRankLinearHilbert}
	The Hilbert series of $A(M(\FF_q^n))$ is equal to $A_n(q,t)$.
\end{cor}
\begin{proof}
	The $q$-exponential generating function of the Hilbert series $h_n(t) = H(A(M(\FF_q^n)), t)$ is the same as the one for the $q$-Eulerian polynomials given in Theorem \ref{thm:swexp}.
\end{proof}
As a corollary, we find a interpretation of the $q$-Eulerian numbers.
\begin{cor}
\[ \eulernom{n}{k}_q = \#\setof{x_{V_{1}}^{\alpha_1} \dots x_{V_{\ell}}^{\alpha_\ell}}{\substack{{V_1\subsetneq \cdots \subsetneq V_\ell\text{ are subspaces of }\FF_q^n} \\ {1\leq \alpha_i \leq \dim V_i - \dim V_{i-1} - 1,\; \sum_i\alpha_i = k}}}\]
\label{qEulerianNos}
\end{cor}
\begin{proof}
By Theorem \ref{thm:fyhilbert} and Corollary \ref{qEulerianNos}, both quantities count $\dim A(M(\FF_q^n))_k$
\end{proof}
\begin{rmk}
\label{rmk:combqhilb}
In the notation of Subsection \ref{subsec:lowrank}, Corollary \ref{qEulerianNos} states that 
\[
 \eulernom{n}{k}_q = \#M_{n,n,k}
\]
\end{rmk}
\begin{rmk}
In the course of proving the results above, we discovered the following recurrence for the $q$-Eulerian polynomials.
\end{rmk}
\begin{prop}
Let $\hilb_n(t)= \hilb(A(M(\mathbb F_q^n)),t)$ denote the Hilbert series of $A(M(\mathbb F_q^n))$, and let $(a;q)_n \coloneqq (1-a)(1-aq)\cdots (1-aq^{n-1})$ be the Pochhammer symbol. Then $h_n$ satisfies the recurrence
\begin{align}
h_n(t) &= \sum_{k=0}^{n-1} \qnom{n}{k}_{q}  h_k(t) \prod_{i=1}^{n-1-k}(t-q^i) \label{rec} \\
&= \sum_{k=0}^{n-1} \qnom{n}{k}_{q} t^{n-1-k}\cdot h_k(t)\cdot (q/t;q)_{n-1-k}. \nonumber
\end{align}
\label{qeulerianrec}
\end{prop}
To the authors' knowledge, the recurrence in proposition \ref{qeulerianrec} does not yet appear in the literature, and it provides a $q$-analogue for the following well-known recurrence for the Eulerian polynomials 
\[
A_n(t) = \sum_{k=0}^{n-1} \binom{n}{k}A_k(t)(t-1)^{n-1-k}.
\]
For a proof of Proposition \ref{qeulerianrec}, see our REU report \cite{report}.
\subsection{Lower rank vector space matroids}
\label{subsec:lowrank}
Next, we find an explicit form for the Hilbert series of lower rank vector space matroids $M_r(\FF_q^n)$ with $r < n$.
The main result of this section is Theorem \ref{cor:linearhilbert}.

We will first give a brief overview of our methodology and set up some notation.
We study the Hilbert series of $A\big(M_r(\FF_q^n)\big)$ by descending induction on the rank $r$; in particular, we consider the differences $\Delta_{n,r,q}(t) \coloneqq \hilb\Big( A\big(M_{r+1}(\FF_q^n), t\big) \Big) - \hilb\Big( A\big(M_{r}(\FF_q^n), t\big) \Big)$ for $1\leq r\leq n$. Write
\[
\Delta_{n,r,q}(t) = a_{n,r,q}^{(r)}t^r + a_{n,r,q}^{(r-1)}t^{r-1} + \cdots +a_{n,r,q}^{(0)}
\]
for $a_{n,r,q}^{(k)}\in \ZZ$. We will show that $a_{n,r,q}^{(k)}$ is a $q$-analogue of the number
\[
\#\setof{\sigma\in F_{n,n-r}}{\exc(\sigma) = r-k}.
\]
where $F_{n,n-r} \coloneqq \setof{ \sigma \in \Sym_n}{\#\fix(\sigma) \geq n-r}$.
In particular, we will express
\[
 a_{n,r,q}^{(k)} = \sum_{i = 0}^r \qnom{n}{i}_{q} D_{i,r-k,q} = \sum_{i = 0}^r \qnom{n}{r-i}_{q} D_{r-i,k-i,q} 
\]
where $\mathcal{D}_{n} \subseteq \Sym_{n}$ is the set of derangements, and $D_{n,k,q}$ is a $q$-analogue of the number
\[\#\setof{\sigma\in \mathcal D_n}{\exc(\sigma) = r-k}.\]
Define
\begin{align*}
N_{n,r}&\coloneqq N_{n,r}(q) \coloneqq \setof{x_{\top}^{\alpha_0}x_{V_1}^{\alpha_1}\cdots x_{V_\ell}^{\alpha_\ell}}{\substack{{\FF_q^n \supsetneq V_1\supsetneq \cdots\supsetneq V_\ell\text{ are subspaces of }\FF_q^n\text{ of rank }\leq r,}\\{\alpha_0\leq r-{\rm dim}(V_1)\text{ and } 1\leq \alpha_i\leq {\rm dim}(V_i) - {\rm dim}(V_{i+1}) - 1}}} \\
M_{n,r,k}&\coloneqq M_{n,k,r}(q) \coloneqq \setof{x_{\top}^{\alpha_0}x_{V_1}^{\alpha_1}\cdots x_{V_\ell}^{\alpha_\ell}\in N_{n,r}}{\deg x_{\top}^{\alpha_0}x_{V_1}^{\alpha_1}\cdots x_{V_\ell}^{\alpha_\ell} = k} \\
T_{n,k,q} &\coloneqq \setof{x_{\top}^{\alpha_0}x_{V_1}^{\alpha_1}\cdots x_{V_\ell}^{\alpha_\ell} \in M_{n,n,k}}{\alpha_0\geq 1} \\
D_{n,k,q}&\coloneqq \#T_{n,k,q}.
\end{align*}
For notational convenience, we suppress the dependence on $q$ in $N_{n,r}(q)$ and $M_{n,r,k}(q)$. By Theorem \ref{thm:fyhilbert}, $\dim \big(A(M_r(\FF_q)) \big)_k = \#M_{n,r,k}$. Note that we have inclusions $M_{n,r,k}\subseteq M_{n,r+1,k}$ and the complement of $M_{n,r,k}$ in $M_{n,r+1,k}$ is the set
\[  M_{n,r+1,k} \setminus M_{n,r,k} = \setof{x_{\top}^i x_{V_1}^{\alpha_1}\cdots x_{V_\ell}^{\alpha_\ell} \in M_{n,r+1,k}}{0\leq i\leq r,\; \; {\rm dim}(V_1) = r - i}  \]
Identifying $V_1 = \FF_q^{r-i}$ we obtain, for each fixed $0\leq i\leq r$, a bijection
\begin{align*} \setof{x_{\top}^i x_{V_1}^{\alpha_1}\cdots x_{V_\ell}^{\alpha_\ell} \in N_{n,r,k}}{{\rm dim}(V_1) = r - i} &\to \setof{V_1\subsetneq \FF_q^n}{{\rm dim}(V_1) = r-i}\times T_{r-i,k-i,q}  \\
x_{\top}^i x_{V_1}^{\alpha_1}\cdots x_{V_\ell}^{\alpha_\ell} &\mapsto (V_1, x_{V_1}^{\alpha_1}\cdots x_{V_\ell}^{\alpha_\ell})
\end{align*}
Hence, summing over possible values of the exponent $i$ of $x_\top$ gives
\begin{equation}
\#(M_{n,k,r+1} \setminus M_{n,k,r}) = \sum_{i = 0}^r \qnom{n}{r-i}_q D_{r-i,k-i,q}. \label{hilbseriesK}
 \end{equation}
We will now give a combinatorial description of $D_{n,k,q}$ in terms of elementary statistics on $\mathfrak S_n$. To do so, we establish some notation. For $\sigma\in \Sym_A$ for $A = \{a_1<\cdots <a_k\}$ an ordered set, let the \emph{reduction} of $\sigma$ be the permutation $\overline{\sigma}$ in $\Sym_k$ such that $\sigma(a_i) = a_{\overline{\sigma}(i)}$. For $\sigma\in \Sym_n$, its \emph{derangement part} ${\rm dp}(\sigma)$ is the reduction of $\sigma$ along its nonfixed points.
The following lemma of Wachs will be essential.
\begin{lem}[\cite{wac} Corollary 3] For all $\gamma\in \mathcal D_k$ and  $n\geq k$,
\[  \sum_{\substack{{\rm dp}(\sigma) = \gamma \\ \sigma\in \Sym_n}}q^{\maj(\sigma)} = q^{\maj(\gamma)}\qnom{n}{k}_q  \]
\label{wachsLemma}
\end{lem}
From this lemma, another useful identity follows.
\begin{cor}
For any integers $n,q,k\geq 0$,
\[
\sum_{\substack{\sigma\in \mathcal D_{n-i} \\ \exc(\sigma) = k}}q^{\maj(\sigma) - \exc(\sigma)}\qnom{n}{n-i}_q = \sum_{\substack{\sigma\in \Sym_n\\ \exc(\sigma) = k \\ \#\fix(\sigma) = i}}q^{\maj(\sigma) - \exc(\sigma)}
\]
\label{waccor}
\end{cor}
\begin{proof}
From Lemma \ref{wachsLemma}, we have the identity
\[ 
\sum_{\substack{\gamma\in \mathcal D_{n-i} \\ \exc(\gamma) = k}}q^{\maj(\gamma) - \exc(\gamma)}\qnom{n}{n-i}_q = \sum_{\substack{\gamma\in \mathcal D_{n-i} \\ \exc(\gamma) = k}}q^{-\exc(\gamma)}\sum_{\substack{\sigma\in \Sym_n \\ {\rm dp}(\sigma) = \gamma}}q^{\maj(\sigma)} = \sum_{\substack{\sigma\in \Sym_n \\ \exc(\sigma) = k \\ \#\fix(\sigma) = i}}q^{\maj(\sigma) - \exc(\sigma)}.\qedhere
 \]
\end{proof}
We now make use of this identity to give a combinatorial interpretation to both $D_{n,k,q}$ and $a_{n,r,q}^{(k)}$. 
\begin{lem}
For $D_{n,k,q}$ as above,
\[
D_{n,k,q} = \sum_{\substack{\sigma\in \mathcal D_n \\ \exc(\sigma) = n-k}}q^{\maj(\sigma) - \exc(\sigma)}
\]
\label{qDerangements}
\end{lem}
\begin{proof}
We proceed by induction on $k$. For $k = 0$, the result is vacuous. For $k>0$, set
\begin{align*}
S_{\alpha_0} &\coloneqq \setof{x_\top^{\alpha_0}x_{V_1}^{\alpha_1}\cdots x_{V_\ell}^{\alpha_\ell} \in M_{n,n,k-1}}{{\rm dim}(V_1) = n-\alpha_0-1} \\
S &\coloneqq M_{n,n,k - 1}.
\end{align*}
Then, the map on monomials taking $x_{\top}^{\alpha_0}x_1^{\alpha_1}\cdots x_\ell^{\alpha_\ell}\mapsto x_{\top}^{\alpha_0-1}x_1^{\alpha_1}\cdots x_\ell^{\alpha_\ell}$ gives an injective map
\[ \varphi\colon T_{n,k,q}\to S. \]
Moreover, $S$ is the disjoint union $S = {\rm Im}(\varphi)\sqcup \coprod_{a\geq 0}S_a$. Considering the choice of the second largest subspace,
\[
\#S_a = \qnom{n}{n-a-1}_q D_{n-a-1,k-a-1,q}
\]
While from Remark \ref{rmk:combqhilb},
\[ \#S = \eulernom{n}{k-1}_q =  \eulernom{n}{n-k}_q \]
where the latter equality follows from Poincar\'e duality for $A\big(M(\FF_q^n)\big)$. Therefore, by induction, 
\begin{align}
D_{n,k,q} &= \#T_{n,k,q} = \#S - \sum_{a\geq 0}\#S_a = \eulernom{n}{n-k}_q - \sum_{b\geq 1} \qnom{n}{n-b}_{q} D_{n-b,k-b,q} \nonumber\\
& =\sum_{\substack{\sigma\in \Sym_n \\ \exc(\sigma) = n-k}}q^{\maj(\sigma) - \exc(\sigma)} - \sum_{b\geq 1}\sum_{\substack{\gamma\in \mathcal D_{n-b} \\ \exc(\gamma) = n-k}}q^{\maj(\gamma) - \exc(\gamma)}\qnom{n}{n-b}_q  \label{inductioneqn}
\end{align}
Then applying Corollary \ref{waccor}, the right-hand side of equation \ref{inductioneqn} can be expanded as
\begin{align*}
\sum_{\substack{\sigma\in \Sym_n \\ \exc(\sigma) = n-k}}q^{\maj(\sigma) - \exc(\sigma)} - \sum_{b\geq 1}\sum_{\substack{\sigma\in \Sym_n \\ \exc(\sigma) = n-k \\ \#\fix(\sigma) = b}}q^{\maj(\sigma) - \exc(\sigma)}=\sum_{\substack{\sigma\in \mathcal D_n \\ \exc(\sigma) = n-k}}q^{\maj(\sigma) - \exc(\sigma)}
\end{align*}
completing the induction and proof of the theorem.
\end{proof}

\begin{lem}
Let $F_{n,k}$ denote the set $F_{n,k} = \setof{\sigma\in \Sym_n}{\#\fix(\sigma)\geq k}$. The difference of Hilbert series $\Delta_{n,r,q}(t)$ is given by
\[ \Delta_{n,r,q}(t) = \hilb\Big( A\big(M_{r+1}(\FF_q^n), t\big) \Big) - \hilb\Big( A\big(M_{r}(\FF_q^n), t\big) \Big) = \sum_{\sigma\in F_{n,n-r}}t^{r-\exc(\sigma)}q^{\maj(\sigma) - \exc(\sigma)}  \]
In particular, the coefficients $a_{n,r,q}^{(k)}$ satisfy
\begin{equation}
a_{n,r,q}^{(k)} = \sum_{\substack{\sigma\in F_{n,n-r} \\ \exc(\sigma) = r-k}} q^{\maj(\sigma) - \exc(\sigma)}\label{qRankHilbFxn} 
\end{equation}
\label{qkernel}
\end{lem}
\begin{proof}
Applying Theorem \ref{qDerangements} and Corollary \ref{waccor} to equation \eqref{hilbseriesK} gives
\begin{align*}
a_{n,r,q}^{(k)}
&= \sum_{i=0}^r \qnom{n}{r-i}_qD_{r-i,k-i,q} 
= \sum_{i = 0}^r \qnom{n}{r-i}_q \sum_{\substack{\sigma\in \mathcal D_{r-i} \\ \exc(\sigma) = r-k}} q^{\maj(\sigma) - \exc(\sigma)} \\
&= \sum_{i=0}^r\sum_{\substack{\sigma\in \Sym_n \\ \#\fix(\sigma) = n-r+i \\ \exc(\sigma) = r-k}} q^{\maj(\sigma) - \exc(\sigma)} \\
&= \sum_{\substack{\sigma\in F_{n,n-r} \\ \exc(\sigma) = r-k}}q^{\maj(\sigma) - \exc(\sigma)}.\qedhere
\end{align*}
\end{proof}
These two lemmas yield the main result.
\begin{proof}[Proof of Theorem \ref{cor:linearhilbert}]
Equation \eqref{qRankHilb} follows from a direct substitution of \eqref{qRankHilbFxn} into the formula
\begin{align*}
	\hilb\big(A(M_r(\FF_q^n),t \big) 
	&= \hilb\big(A(M_{r+1}(\FF_q^n)),t\big)-\Delta_{n,r,q}(t) \\
	&= \cdots 
	= \hilb\big( A( M(\FF_q^n)),t\big) - \sum_{j = r}^{n-1} \Delta_{n,j,q}(t) \qedhere
\end{align*}
\end{proof}
When $r = n-1$, the Hilbert series assumes a more pleasing form.
\begin{cor}
If $r = n-1$, the Hilbert series of $A\big( M_{n-1}(\FF_q^n) \big)$ is
\[
\hilb\Big( A\big( M_{n-1}(\FF_q^n) \big), t\Big) = \sum_{\sigma\in \mathcal D_n}q^{\maj(\sigma) - \exc(\sigma)}t^{\exc(\sigma)-1}
\]
\end{cor}
\begin{proof}
For the case $r = n-1$, the coefficient of $t^k$ in \eqref{qRankHilb} can be simplified as follows.
\begin{align*}
\sum_{\substack{\sigma\in \Sym_n \\ \exc(\sigma) = k}}q^{\maj(\sigma) - \exc(\sigma)} &- \sum_{\substack{\sigma\in F_{n,1} \\ \exc(\sigma) = n-k-1}}q^{\maj(\sigma) - \exc(\sigma)} \\
= \sum_{\substack{\sigma\in \Sym_n \\ \exc(\sigma) = n-k-1}}q^{\maj(\sigma) - \exc(\sigma)} &- \sum_{\substack{\sigma\in F_{n,1} \\ \exc(\sigma) = n-k-1}}q^{\maj(\sigma) - \exc(\sigma)} \\
= \sum_{\substack{\sigma\in \mathcal D_n \\ \exc(\sigma) = n-k-1}}q^{\maj(\sigma) - \exc(\sigma)} &
\end{align*}
Then, 
\[ \hilb\Big( A\big( M_r(\FF_q^n) \big), t\Big) = \sum_{\sigma\in \mathcal D_n}q^{\maj(\sigma) - \exc(\sigma)}t^{n-1-\exc(\sigma)} =  \sum_{\sigma\in \mathcal D_n}q^{\maj(\sigma) - \exc(\sigma)}t^{\exc(\sigma)-1} \]
where the last equality follows from Poincar\'e duality of $A(M_{n-1}(\FF_q^{n}))$.
\end{proof}
\begin{rmk}
The proof presented in the previous section can be reformulated in terms of strong maps of Chow rings. Namely, consider the graded, surjective ring homomorphisms
\[ \pi_{n,r,q}\colon A(M_{r+1}\big(\FF_q^n)\big)\to A\big(M_r(\FF_q^n)\big) \]
defined by taking variables $x_{V}\in A(M_{r+1}\big(\FF_q^n)\big)$ to zero if $\dim(V) = r+1$ and to the corresponding variable $x_V\in A\big(M_r(\FF_q^n)\big)$ otherwise. Then, if $K_{n,r,q} = \ker(\pi_{n,r,q})$, additivity of Hilbert series gives 
\[ \hilb(K_{n,r,q},t) = \hilb\Big( A\big(M_{r+1}(\FF_q^n), t\big) \Big) - \hilb\Big( A\big(M_{r}(\FF_q^n), t\big) \Big) = \Delta_{n,r,q}(t) \]
Therefore, Lemma \ref{qkernel} gives a formula for the Hilbert series of the kernel of the above so-called ``strong maps'' of Chow rings.
\end{rmk}
\begin{rmk}
Note that the characterization of the Hilbert series of $A(M_r(\FF_q^n))$ for $r = n-1,n$ together with the results of \cite{ahk} give an alternate proof of the unimodality and symmetry of the polynomials
\[
\sum_{\sigma\in \mathfrak S_n}q^{\maj(\sigma) - \exc(\sigma)}t^{\exc(\sigma)}\text{ and }\sum_{\sigma\in \mathcal D_n}q^{\maj(\sigma) - \exc(\sigma)}t^{\exc(\sigma)-1}.
\]
However, it should be noted that in \cite{qgamma}, Shareshian and Wachs prove more general statements. Namely, they prove that the coefficients of the above polynomials are $q$-unimodal and, in fact, $q$-$\gamma$-nonnegative. That is, a difference of consecutive coefficients lies in $\NN[q]$ as a polynomial in $q$, and moreover, its $\gamma$-vector has coordinates in $\NN[q]$. See Theorems 4.4 and 6.1 of \cite{qgamma} for more explicit formulae and a proof.
\end{rmk}

\section{Charney-Davis quantities of vector space matroids}
\label{sec:charneydavis}
The main result of this section is a proof of Theorem \ref{thm:linearcd}, which gives two formulas for the Charney-Davis quantity of $A\big(M_r(\mathbb F_q^n)\big)$, one in terms of determinants and one in terms of $q$-tangent-secant numbers.
We prove the formula that is in terms of determinants immediately; we will prove the formula in terms of $q$-tangent-secant numbers later.
\begin{proof}[Proof of Theorem \ref{thm:linearcd} (b)]
If $r = 1$, then $\hilb\Big(A\big( M_r(\FF_q^n) \big),t\Big) = 1$, and the theorem follows trivially. Now suppose that $r>1$ is odd, and let $\CD(n,r) = \hilb\Big(A\big(M_r(\FF_q^n)\big),-1\Big)$ be the unsigned Charney-Davis quantity of $A\big(M_r(\FF_q^n)\big)$. Substituting $t = -1$ into Theorem \ref{thm:fyhilbert}, the formula for the Hilbert series from \cite{fy} is
\[
\CD(n,r) = 1+\sum_{\substack{{\rr,\, r_k<r}\\{\forall i, r_i - r_{i-1}\text{ is even}}}}(-1)^{|\rr|}\prod_{i = 1}^{|\rr|} \qnom{n-r_{i-1}}{r_i - r_{i-1}}_{q}.
\]
where $|\rr|$ is the number of entries in the tuple $\rr$.
Breaking into cases based on whether $\rr = (r_1<\cdots<r_k)$ has $r_k = r-1$, we get a decomposition of the above as
\[
\left\{1+\sum_{\substack{{\rr,\, {r_k<r-2}}\\{\forall i, r_i - r_{i-1}\text{ is even}}}}(-1)^{|\rr|}\prod_{i = 1}^{|\rr|} \qnom{n-r_{i-1}}{r_i - r_{i-1}}_q\right\}
+ \left\{\sum_{\substack{{\rr,\, {r_k=r-1}}\\{\forall i, r_i - r_{i-1}\text{ is even}}}}(-1)^{|\rr|}\prod_{i = 1}^{|\rr|} \qnom{n-r_{i-1}}{r_i - r_{i-1}}_{q}\right\}
\]
where the former term is $\CD(n,r-2)$ and the latter we denote by $T_{n,q}(r-1)$. Then, considering terms in the sum with $r_{k-1} = b$, one obtains the recurrence
\[
T_{n,q}(2a) = -\sum_{b=0}^{a-1} \qnom{n-2b}{2a-2b}_qT_{n,q}(2b) \;\;\text{with initial condition }\;\;T_{n,q}(0) = 1
\]
Solving this linear recurrence with Cramer's rule gives 
\begin{equation}
\label{Tform}
T_{n,q}(2a) = (-1)^a \det\left( \begin{array}{ccccc} \qnom{n}{2}_q & 1 & 0 & \cdots & 0\\ \qnom{n}{4}_q & \qnom{n-2}{2}_q & 1 & \cdots & 0\\
\vdots & \vdots & \vdots & \ddots & \vdots \\
\qnom{n}{2a-2}_q & \qnom{n-2}{2a-4}_q & \qnom{n-4}{2a-6}_q& \cdots& 1\\ \qnom{n}{2a}_q & \qnom{n-2}{2a-2}_q & \qnom{n-4}{2a-4}_q& \cdots& \qnom{n-2a+2}{2}_q \end{array} \right)
\end{equation}
Rewriting the determinant in \eqref{Tform} by pulling out common factors in the numerator, resp. denominators, of each column, resp.\ row, gives
\begin{align*}
T_{n,q}(2a) &= (-1)^a\frac{[n]_q!}{[n-2a]_q!}\det\left( \begin{array}{ccccc} \frac{1}{[2]_q!} & 1 & 0 & \cdots & 0\\ \frac{1}{[4]_q!} & \frac{1}{[2]_q!} & 1 & \cdots & 0\\
\vdots & \vdots & \vdots & \ddots & \vdots \\
\frac{1}{[2a-2]_q!} & \frac{1}{[2a-4]_q!} & \frac{1}{[2a-6]_q!} & \cdots& 1\\ \frac{1}{[2a]_q!} & \frac{1}{[2a-2]_q!} & \frac{1}{[2a-4]_q!} & \cdots& \frac{1}{[2]_q!} \end{array} \right) \\
&= (-1)^a\frac{[n]_q!}{[n-2a]_q!}\Delta_{a,q}
\end{align*}
Then, the unsigned Charney-Davis quantity for odd $r$ is
\begin{align*}
\CD(n,r) &= \CD(n,r-2)+T_{n,q}(2k) = \cdots = \CD(n,1)+\sum_{a = 1}^{\frac{r-1}{2}} T_{n,q}(2a)\\
& = 1+[n]_q!\sum_{a = 1}^{\frac{r-1}{2}} \frac{(-1)^a}{[n-2a]_q!}\Delta_{a,q}.
\end{align*}
Then, the result follows by multiplication by the appropriate sign.
\end{proof}
\begin{ex}
For the case $n = r = 5$, Theorem \ref{thm:linearcd} becomes the following identity
\[ q^8+2q^7+3q^6+4q^5+3q^4+2q^3+q^2 = 1+[5]_q!\left[ -\frac{1}{[3]_q!}\det\left(\frac{1}{[2]_q!}\right) +\det\left(\begin{array}{cc} \frac{1}{[2]_q!} & 1 \\ \frac{1}{[4]_q!} & \frac{1}{[2]_q!} \end{array}\right) \right] \]
which one can directly verify.
\end{ex}
\begin{rmk}
\label{rmk:evencd}
For even $r$, Theorem 6.19 of \cite{ahk} implies the Hilbert series of $A(M_r(\FF_q^n))$ is symmetric of even degree. Consequently, $\hilb\big( A(M_r(\FF_q^n)),-1 \big) = 0$ and the Charney-Davis quantity vanishes.
\end{rmk}
Having the determinantal formula above, we now work towards a more compact formula using the $q$-tangent/secant numbers.
\begin{prop}
Let $E_{n,q}$ denote the $n$-th $q$-tangent/secant number. The following identities hold:
\[ E_{2n,q} = (-1)^n [2n]_q! \Delta_{n,q} \]
\[ E_{2n+1, q} = \CD(2n+1, 2n+1) = 1 + [2n+1]_q! \sum_{a = 1}^n \frac{(-1)^a}{[2n-2a + 1]_q!} \Delta_{a,q} \]
\label{prop:qSecTan}
\end{prop}
\begin{proof}
Let
\[ \mathcal E_{2n,q} \coloneqq (-1)^n [2n]_q! \Delta_{n,q} \]
\[ \mathcal E_{2n+1, q} \coloneqq \CD(2n+1, 2n+1) = 1 + [2n+1]_q! \sum_{a = 1}^n \frac{(-1)^a}{[2n-2a + 1]_q!} \Delta_{a,q}. \]
Consider the generating functions
\[
F(t) = \sum_{n\geq 0}\mathcal E_{2n,q}\frac{t^{2n}}{(q;q)_{2n}} \;\;\;\;\; \text{ and } \;\;\;\;\; G(t) = \sum_{n\geq 0}\mathcal E_{2n+1,q}\frac{t^{2n+1}}{(q;q)_{2n+1}}
\]
It suffices to show $F(t) = \sech_q(t)$ and $G(t) = \tanh_q(t)$. Observe that by expanding by minors in the first column, $\Delta_{n,q}$ satisfies the recurrence
\[
\Delta_{n,q} = \sum_{k = 1}^n\frac{(-1)^{k+1}}{[2k]_q!}\Delta_{n-k,q}
\]
Then since $(q;q)_{2n} = \frac{[n]_q!}{(1-q)^n}$,
\begin{align*}
F(t) &= \sum_{n\geq 0}(-1)^n\big( t(1-q) \big)^{2n}\Delta_{n,q} = 1+\sum_{n\geq 1}(-1)^n\big( t(1-q) \big)^{2n}\sum_{k=1}^n\frac{(-1)^{k+1}}{[2k]_q!}\Delta_{n-k,q}\\
& = 1+\sum_{r \geq 0} \sum_{k\geq 1}(-1)^{r+1}\Delta_{r,q}\frac{\big( t(1-q) \big)^{2(r+k)}}{[2k]_q!}\\
&= 1+\left(\sum_{k\geq 1}\frac{\big( t(1-q) \big)^{2k}}{[2k]_q!}\right)\left( \sum_{r\geq 0}(-1)^{r+1}\Delta_{r,q} \big( t(1-q) \big)^{2r} \right) \\
&= 1-\left(\sum_{k\geq 1}\frac{t^{2k}}{(q;q)_{2k}}\right)F(t)= 1-(\cosh_q(t)-1)F(t)
\end{align*}
Therefore, solving for $F(t)$ gives 
\[
F(t) = 1/\cosh_q(t) = \sech_q(t)
\]
Since $F(t) = \sech_q(t)$ as power series in $\QQ(q)[\![t]\!]$, it follows that $\mathcal E_{2n,q} = E_{2n,q}$. Now consider $G(t)$. Set $\Delta_{0,q} = 1$. We have
\begin{align*}
G(t) &= \sum_{n\geq 0}\left( [2n+1]_q! \sum_{a = 0}^n \frac{(-1)^a}{[2n-2a + 1]_q!} \Delta_{a,q} \right)\frac{t^{2n+1}}{(q;q)_{2n+1}}   \\
&= \sum_{n\geq 0}  \sum_{a = 0}^n \frac{(-1)^a\Delta_{a,q}}{[2n-2a + 1]_q!}  {\big(t(1-q)\big)^{2n+1}}   \\
&= \sum_{k\geq 0}\sum_{a\geq 0} \frac{(-1)^a\Delta_{a,q}}{[2k + 1]_q!}  \big(t(1-q)\big)^{2(a+k)+1}  \\
&= \left(\sum_{k\geq 0} \frac{t^{2k+1}}{(q;q)_{2k+1}}\right)\left( \sum_{a\geq 0} (-1)^a\Delta_{a,q}t^{2a} \right) \\
&= \sinh_q(t)\sech_q(t) = {\rm tanh}_q(t)\qedhere
\end{align*}
\end{proof}
\begin{rmk}
With notation as in the proof above, equation (2.6) of \cite{stanley:altPerms} immediately implies that $\mathcal E_{2n,q} = E_{2n,q}$. See equation (2.7) of the same article for a determinantal formula for $E_{2n+1,q}$ and other formulae.
\end{rmk}
\begin{rmk}
Proposition \ref{prop:qSecTan} implies that the numbers $E_{n,q}$ are the $q$-secant and $q$-tangent numbers studied in \cite{foata} and \cite{jv}. In particular, we have
\[ E_{n,q} = \sum_{\sigma\in \mathfrak I_n}q^{\exc(\sigma)} \]
where $\mathfrak{I}_n$ denotes the number of alternating permutations of size $n$. 
\end{rmk}
Theorem \ref{thm:linearcd}(a) now follows from Thm \ref{thm:linearcd}(b) and Prop \ref{prop:qSecTan}.

\section{Invariants of uniform matroids}
\label{sec:uniform}
Recall that the uniform matroid $U_{n,r}$ is the matroid whose independent sets consist of all subsets of $[n]$ of cardinality at most $r$. 
 Theorem \ref{thm:fyhilbert} gives a formula for the Hilbert series of $A\big(M(\FF_q^n)\big)$,
\[
\hilb\Big(A\big(M_r(\FF_q^n)\big),t\Big) = 1+\sum_{\rr} \prod _{i=1} ^{|\rr|} \frac{t(1-t^{r_i-r_{i-1}-1})}{1-t}\qnom{n-r_{i-1}}{r_i - r_{i-1}}_q
\]
where the sum is over all tuples of dimensions $\rr = (0=r_0<r_1<\cdots<r_{|\rr|}\leq r)$. In particular, when $q = 1$, formula above specializes to what Theorem \ref{thm:fyhilbert} gives for $\hilb\Big(A(U_{n,r}),t\Big)$. From this it follows that any invariant of $A(U_{n,r})$ that can be computed in terms of its Hilbert series can be computed by instead considering the corresponding invariant of $A(M_r(\FF_q^n))$ and setting $q = 1$. We record a number of results obtained this way below.
\begin{thm}[see Theorem \ref{cor:linearhilbert}] \label{thm:linearhilbert}
For $r = 0, 1, \ldots, n$ and $F_{n,k}:=\setof{\sigma\in \Sym_n}{\#\fix(\sigma) \geq k}$, the Hilbert series of $A\big( U_{n,r} \big)$ is given by
\[ \hilb\big( U_{n,r},t \big)
= \sum_{\sigma\in \Sym_n}t^{\exc(\sigma)} - \sum_{j=r}^{n-1} \sum_{\sigma\in F_{n,n-j}}t^{r-\exc(\sigma)}  \]
In particular, if $r = n$, the Hilbert series of $A(U_{n,n})$ is the $n$-th Eulerian Polynomial and if $r = n-1$, the Hilbert series of $A\big( U_{n,n-1} \big)$ is
\[
\hilb\Big( A\big( U_{n,n-1} \big), t\Big) = \sum_{\sigma\in \mathcal D_n}t^{\exc(\sigma)-1}
\]
\end{thm}

\begin{thm}[see Theorem \ref{thm:linearcd}]
\label{CDuniform}
For odd $r$, the Charney-Davis quantity for the uniform matroid, $U_{n,r}$, of rank $r$ and dimension $n$ is 
	\[  \sum_{k=0}^{\frac{r-1}{2}}\binom{n}{2k}E_{2k}  \]
	where $E_{2\ell}$ is the $\ell$-th secant number, i.e.
	\[ \sech(t) = \sum_{\ell\geq 0}E_{2\ell}\frac{t^{2\ell}}{(2\ell)!} \]
\end{thm}

\begin{rmk}
For $r = n$ odd, a standard recurrence shows 
\[  \sum_{k=0}^{\frac{n-1}{2}}\binom{n}{2k}E_{2k} = E_{n}    \]
In particular, Theorem \ref{CDuniform} specializes to those in page 275 of \cite{rcharney} and page 52 of \cite{hshell}.
\end{rmk}

\begin{rmk}
	Those interested in the $\gamma$-polynomial of $A(U_{n,r})$ for $r = n, n-1$ should see Theorem 11.1 of \cite{pr} and Theorem 4.1 of \cite{athan}.
	The former gives the $\gamma$-vector of $A(U_{n,n})$ in the context of the $\gamma$-vector of the permutohedron. 
Since $\hilb\big( A(U_{n,n-1}),t \big)$ is the local $h$-vector of the barycentric subdivision of the permutohedron, Athanasiadis' survey \cite{athan} gives the analogous interpretation of the $\gamma$-vector of $\hilb\big(A(U_{n,n-1}),t\big)$.
\end{rmk}
\section{Conjectures and future work}
\label{sec:conjectures}
Our data points to a possible relationship between order complexes and Chow rings.
Let $\Delta(P)$ be the order complex of a poset $P$, and for any simplicial complex $S$, denote the $h$-polynomial of $S$ by
\[ \hpoly(S, t) \coloneqq \sum_{i = 0}^{\olddim(S)} f_{i-1}(x-1)^{\olddim(S) -i} \]
	where $f_j$ is the number of $j$-dimensional faces of $S$ and $f_{-1} = 1$ by convention.

\begin{prop}[\cite{pk} Theorem 9.1, \url{https://oeis.org/A008292}]
For all $n\geq1$,
\[h\big(\Delta(L(U_{n,n})), t\big) = \hilb\big(A(U_{n,n}),t\big)  \]
\end{prop}

The corresponding statement for the uniform matroids $U_{n,r}$ with $r < n$ has small counterexamples, but can be modified as follows.
\begin{conj}
For $r<n$, we have
    \[ \hpoly\big(\Delta(L(U_{n,r})), t\big) = t^{2} \sum_{i = 1}^{r} \binom{n-i-1}{r-i} \hilb(A(U_{n,i}), t). \]
	\label{conj:order-complex-hpoly}
\end{conj}

Since it is relatively simple to compute the $f$-vector of $\Delta(\LL(U_{n,r}))$, this would also give a formula for $\hilb(A(U_{n,i+1}),t)$.

\begin{rmk}
Conjecture \ref{conj:order-complex-hpoly} is equivalent to the equality $F_n(t,u) = H_n(t,u+1)$ for the polynomials
\begin{align*}
F_n(t,u) &= \sum_{r = 0}^{n-2} \hpoly( \Delta(\mathcal L(U_{n,r+1}\setminus\{\top,\bot\})),t)u^{n-2-r} \\
H_n(t,u) &= \sum_{r=0}^{n-2}\hilb(A(U_{n,r+1}),t)u^{n-2-r}
\end{align*}
\end{rmk}

For more conjectures and some other results pertaining to Chow rings of general \atomic lattices, see \cite{report}.

\section*{Acknowledgments} 
This research was carried out as part of the 2017 summer REU program at the School of Mathematics, University of Minnesota, Twin Cities, and was supported by NSF RTG grant DMS-1148634 and by NSF grant DMS-1351590. The authors would like to thank Victor Reiner, Pavlo Pylyavskyy, and Benjamin Strasser for their mentorship and support.
\nocite{*}
\printbibliography
\end{document}